\documentclass[a4paper,12pt]{amsart}

\usepackage[T1]{fontenc}
\usepackage[utf8]{inputenc}

\usepackage[english]{babel}

\usepackage{amsmath, amssymb, amsthm, amsfonts}
\usepackage{mathrsfs}
\usepackage{tikz}
\usetikzlibrary{arrows}
\usetikzlibrary{matrix}
\usetikzlibrary{cd}
\usepackage[left=2.5cm, right=2.5cm, top=2cm]{geometry}
\usepackage{enumerate}
\usepackage{hyperref}

\theoremstyle{plain}
\newtheorem{thm}{Theorem}
\newtheorem{lem}{Lemma}

\theoremstyle{definition}
\newtheorem{defi}{Definition}
\newtheorem{rema}{Remark}

\newcommand{\Z}{\mathbb{Z}^+}
\newcommand{\C}{\mathbb{C}}
\newcommand{\Q}{\mathbb{Q}^+}
\newcommand{\Pfin}{\mathcal{P}^{fin} \left( \mathbb{Z}^+ \right)}

\begin{document}

\title{On a theorem of Hildebrand}
\author{Carsten Dietzel}
\email{carstendietzel@gmx.de}
\address{Institute of algebra and number theory, University of Stuttgart, Pfaffenwaldring 57, 70569 Stuttgart, Germany}
\date{\today}

\begin{abstract}
We prove that for each multiplicative subgroup $A$ of finite index in $\Q$, the set of integers $a$ with $a, a+1 \in A$ is an IP-set. This generalizes a theorem of Hildebrand concerning completely multiplicative functions taking values in the $k$-th roots of unity.
\end{abstract}

\maketitle

A theorem of Hildebrand \cite[Theorem 2]{hildebrand1991}, which was essential in answering a question of Lehmer, Lehmer and Mills on consecutive power residues (\cite{lehmer_lehmer_mills}) can be formulated as follows:

\begin{thm}[Hildebrand]
Fix some $k \in \Z$. If $f: \Z \to \C $ is a completely multiplicative function (i.e. $f(mn) = f(m)f(n)$ for all $m,n \in \Z$) taking its values in the $k$-th roots of unity then the set of $a \in \Z$ fulfilling $f(a) = f(a+1) = 1$ is nonempty.
\end{thm}

\begin{rema}
Hildebrand actually proved more, i.e. there is a constant $c(k)$, independent on the specific multiplicative function $f$, and an $a \in \Z$ such that $a \leq c(k)$ and $f(a) = f(a+1) = 1$. However, by a standard compactness argument, these versions can be seen to be equivalent.
\end{rema}

It makes sense to restate Hildebrand's result as follows:

\begin{thm}[Hildebrand] \label{thm:main_hildebrand}
Let $A \leq \Q$ be a (multiplicative) subgroup such that $\Q / A$ is cyclic of finite order. Let $A^{\ast} := A \cap \Z$. Then $A^{\ast} \cap (A^{\ast} - 1)$ is nonempty.
\end{thm}

The original proof made use of analytic methods and was rather long. We will give a short elementary proof of a more general theorem.

However, before we can state (and prove) our generalization we need some notation and the set-theoretical version of Hindman's theorem:

We denote by $\Pfin$ the set of finite, non-empty subsets of $\Z$.

For $A,B \in \Pfin$ write $A \prec B$ iff $\max A < \min B$.

Furthermore, for a sequence $A_1 \prec A_2 \prec \ldots$ in $\Pfin$, we define
\[
FU((A_i)_{i \in \Z}) = \left\{ \bigcup_{i \in I} A_i \, : \, I \subseteq \Z, 0 < \vert I \vert < \infty \right\}
\]

Similarly, for a sequence $a_1, a_2, \ldots$ in $\Z$, we define
\[
FS((a_i)_{i \in \Z}) = \left\{ \sum_{i \in I} a_i \, : \, I \subseteq \Z, 0 < \vert I \vert < \infty \right\}
\]

We call a set $M \subseteq \Z$ an \emph{IP-set} (\cite[Definition 16.3]{hindman_strauss}) if there is a sequence $a_1,a_2, \ldots$ in $\Z$ such that $FS((a_i)_{i \in \Z}) \subseteq M$.

Then Hindman's theorem on partitions of $\Pfin$ (\cite[Corollary 5.17]{hindman_strauss}) can be stated as follows:

\begin{thm}[Hindman] \label{thm:hindman}
For any finite partition $\Pfin = M_1 \uplus M_2 \uplus \ldots M_n$ there are sets $A_1 \prec A_2 \prec \ldots $ and $1 \leq j \leq k$ such that
\[
FU((A_i)_{i \in \Z} \subseteq M_j.
\]
\end{thm}

We can now state our generalization of Hildebrand's theorem:

\begin{thm} \label{thm:main}
Let $A \leq \Q$ be a (multiplicative) subgroup of finite index. Let $A^{\ast} := A \cap \Z$. Then $A \cap (A - 1)$ is an IP-set.
\end{thm}

Hildebrand's proof of \autoref{thm:main_hildebrand} is an application of Ramsey's theorem on $\gcd$-sequences, i.e. sequences where the consecutive differences are the $\gcd$s of the corresponding terms. We will use a similar concept:

\begin{defi}
For a sequence $s_n$ and a finite subset $A \subset \Z$, set
\[
s_A := \sum_{n \in A}~ s_n.
\]
A \emph{block-divisible sequence} is a strictly decreasing sequence $s_n$ in $\Z$ such that for $A,B \in \Pfin$, $s_A$ divides $s_B$ whenever $A \prec B$.
\end{defi}

For our proof, \emph{any} block-divisible sequence will work. Thus, we only need to confirm the existence of block-divisible sequences:

\begin{lem}
There is a block-divisible sequence in $\Z$.
\end{lem}

\begin{proof}
We construct a sequence as follows:

\begin{align*}
s_0 & := 1 \\
s_{n+1} & := \prod_{A \subseteq \{0,\ldots,n\},A \neq \varnothing} ~ s_A.
\end{align*}

Ignoring the $s_0$ at the beginning, we end up with a strictly increasing sequence fulfilling the desired divisibility condition.
\end{proof}

Now we can show our main result:

\begin{proof}[Proof of \autoref{thm:main}]
Let $N_i^{\prime}$ ($1 \leq i \leq k$) be the (multiplicative) cosets of $A$ in $\Q$.

These give a finite partition $\Z = N_1 \uplus N_2 \uplus \ldots \uplus N_k$ where $N_i = N_i^{\prime} \cap \Z$.

Define a partition $\Pfin = M_1 \uplus M_2 \uplus \ldots \uplus M_k$ by declaring $A \in M_i :\Leftrightarrow s_A \in N_i$

By \autoref{thm:hindman} there is a sequence $A_1 \prec A_2 \prec \ldots$ such that $FU(A_1,A_2,\ldots)$ is contained in one $N_i$ for some $1 \leq i \leq k$.

By the definition of block-divisibility, $s_{A_1}$ divides $s_A$ for all $A \in FU(A_2,A_3\ldots)$ and, consequently, for all $A \in FU(A_1,A_2\ldots)$, too.

Thus, defining $b_i := s_{A_i}$, the members of $FS(b_1,b_2,\ldots)$ all lie in the same coset of $A$ and are divisible by $b_1$. Therefore, setting $a_i := \frac{b_i}{b_1}$, one has

\[
FS(a_1,a_2,\ldots) = FS(1,a_2,a_3,\ldots) \subseteq A^{\ast}.
\]

Furthermore, $FS(1,a_2,a_3,\ldots) = FS(a_2,a_3,\ldots) \cup FS(a_2,a_3,\ldots)+1 \in A^{\ast}$.

We conclude that $FS(a_2,a_3,\ldots) \subseteq A^{\ast} \cap (A^{\ast} - 1)$.

\end{proof}

\paragraph*{Comment 1}

We use the terminology of \autoref{thm:main} to summarize the state of the art concerning possible generalizations:

There are (multiplicative) subgroups $A$ of arbitrary even index in $\Q$ such that $A^{\ast} \cap (A^{\ast}-1) \cap (A^{\ast}-2)$ is empty, as has been shown by Lehmer and Lehmer (\cite[p.103]{lehmer_lehmer}).

Graham (\cite{graham}) proved that there are subgroups of arbitrary index in $\Q$ such that $A^{\ast} \cap \ldots \cap (A^{\ast}-3)$ is empty.

However, if $\Q/A$ is of odd order $k$ it is still an open question if $A^{\ast} \cap (A^{\ast}-1) \cap (A^{\ast}-2)$ is necessarily nonempty. Only in the case $k = 3$ this set is known to be always nonempty as has been shown computationally by Lehmer, Lehmer, Mills and Selfridge (\cite{llms}). Maybe the combinatorial methods presented in this article could help resolving this problem!

\paragraph*{Comment 2}

Some ideas shown in this article are based on notes of the author (\cite{dietzel}) which have not been submitted to any journal.

\bibliographystyle{halpha}

\end{document}